\documentclass[A4]{article}

\usepackage[dvips]{graphicx}
\usepackage[cmex10]{amsmath}
\usepackage{amsfonts,amsbsy,amssymb,amstext,mathrsfs,latexsym,color,url, amsthm, a4wide}

\newtheorem{theorem}{Theorem}[section]

\newtheorem{proposition}[theorem]{Proposition}
\newtheorem{corollary}[theorem]{Corollary}

\newtheorem{definition}[theorem]{Definition}

\newtheorem{example}[theorem]{Example}

\newcommand{\bitem}{\begin{itemize}}
\newcommand{\eitem}{\end{itemize}}
\newcommand{\benum}{\begin{enumerate}}
\newcommand{\eenum}{\end{enumerate}}
\newcommand{\beq}{\begin{equation}}
\newcommand{\eeq}{\end{equation}}

\newcommand{\norm}[1]{\|#1\|}
\newcommand{\spann}{\mbox{\rm span}}

\newcommand{\supp}{{\text{\rm supp}} \,}

\newcommand{\Id}{\mbox{\rm Id}}

\def\NN{\mathbb{N}}

\def\RR{\mathbb{R}}

\def\cF{{\mathcal{F}}}

\newcommand{\gk}[1]{#1}
\newcommand{\ah}[1]{#1}
\newcommand{\fk}[1]{#1}

\newcommand{\ahh}[1]{#1}

\begin{document}
\renewcommand{\thefootnote}{\fnsymbol{footnote}}
 \footnotetext[1]{
   Department of Mathematics, University of Missouri
   }
 \footnotetext[2]{
   Hausdorff Center for Mathematics, Universit{\"a}t Bonn
   }
\footnotetext[3]{
    Institute of Mathematics, University of Osnabr\"uck
    }
\renewcommand{\thefootnote}{\arabic{footnote}}
\title{Optimally Sparse Frames}

\author{Peter~G.~Casazza \footnotemark[1], Andreas Heinecke\footnotemark[1], Felix Krahmer\footnotemark[2], and Gitta~Kutyniok\footnotemark[3] \\[3ex]
{\it\normalsize Dedicated to the memory of Nigel J. Kalton, who was a great person, friend, and mathematician.}
}

\maketitle

\begin{abstract}
Frames have established themselves as a means to derive redundant, yet stable decompositions of a signal for
analysis or transmission, while also promoting sparse expansions.
However, when the signal dimension is large, the computation of the frame measurements of a signal typically requires a large
number of additions and multiplications, and this makes a frame decomposition intractable in applications with limited computing
budget. To address this problem, in this paper, we \gk{focus on frames in finite-dimensional Hilbert spaces and} introduce
sparsity \fk{for} \gk{such} frame\fk{s} as a new paradigm. In our terminology, a
{\em sparse frame} is a frame whose elements have a sparse representation in an orthonormal basis, thereby enabling
low-complexity frame decompositions. To introduce a precise meaning of {\em optimality}, we take the
sum of the numbers of vectors needed of this orthonormal basis when expanding each frame vector as sparsity measure. We then
analyze the recently introduced algorithm {\em Spectral Tetris} for construction of unit norm tight frames and prove
that the tight frames generated by this algorithm are in fact optimally sparse with respect to the standard unit vector
basis. Finally, we show that even the generalization of Spectral Tetris for the construction of unit norm frames associated
with a given frame operator produces optimally sparse frames.
\end{abstract}

\section{Introduction}
Frames are nowadays a standard methodology in applied mathematics, computer science, and engineering
when redundant, yet stable expansions are required. Examples include sampling theory \cite{E03}, data quantization
\cite{BP07,BLPY10}, quantum measurements \cite{EF02}, coding \cite{BDV00,SH03}, image processing \cite{CD02,KL10},
wireless communication \cite{HBP01,HP02,S01}, time-frequency analysis \cite{DHRS03,WES05}, speech recognition \cite{BCE06},
and bioimaging \cite{CK08}; see also \cite{KC07a,KC07b} for a beautiful survey and further
references. The typical application exploits the decomposition of a signal $x \in \RR^n$ into its frame components, which requires
computation of the frame measurements, i.e., the inner products between the signal $x$ and the frame vectors $(\varphi_i)_{i=1}^N$, say.
However, if the dimension $n$ of the ambient space is large and the frame vectors have `many' non-zero entries, the computational
complexity of the computation of the frame measurements might be high; in fact, for applications with constraints
on the available computing power and bandwidth for data processing, computing the frame measurements and hence the frame decomposition might be intractable.

In this paper, \ah{we} \gk{focus on frames in finite-dimensional Hilbert spaces and} tackle this problem by constructing
frames which have very few non-zero entries, thereby reducing
the number of required additions and multiplications when computing frame measurements significantly. This viewpoint
can be also slightly generalized by assuming that there exists a unitary transformation mapping the frame into
one having this `sparsity property'. Sparsity of fusion frames, which were introduced in \cite{CKL08} as a mathematical
framework for distributed processing thereby going beyond frame theory, was already defined in \cite{CCHKP10} as a concept.
However, the paradigm we aim for in this paper differs from the one introduced in \cite{CCHKP10} for fusion frames when restricting
to the case of frames, since we here aim for an overall sparsity of the frame.

Frame constructions have a long history; browsing through the literature, however, it becomes evident that all
constructions for unit norm tight frames -- those frames most advantageous for applications -- only produce
such frames for very special cases such as harmonic frames, see also \cite{Cas04,CL06}. Very recently, a \fk{significant}
\gk{advance in} the construction of unit norm tight frames was achieved through the introduction of the so-called
{\em Spectral Tetris} algorithm in \cite{CFMWZ09}. For most combinations of the number of frame vectors and the dimension
of the ambient space, this procedure indeed generates a unit norm tight frame. An extension of Spectral Tetris to construct
unit norm frames with prescribed frame operator \ahh{if its eigenvalues are greater or equal to two}
was introduced in \cite{CCHKP10} to allow additional flexibility in the design process.

In this paper we show that \ahh{the} unit norm frames which this extended Spectral Tetris algorithm
generates are {\em optimally sparse} in the sense of the total number of non-zero entries in the frame vectors,
\ahh{provided that Spectral Tetris is performed after ordering the prescribed eigenvalues in an appropriate way}. We
also explicitly determine the exact minimum value of the non-zero entries.
%This result then immediately provides an
%explicit construction for optimally sparse unit norm frames with a prescribed frame operator.
Along the way, we
introduce {\em block decompositions} as a novel structural property of unit norm frames, which we anticipate to be
useful also in other settings.

%**********************************************************************************************************************************
%**********************************************************************************************************************************

\subsection{Main Contribution}

Our main contribution is hence two-fold: Firstly, we introduce sparsity of a frame as a novel paradigm in frame theory.
More precisely, we introduce the notion of a sparse frame as well as a sparsity measure for such frames, thereby
allowing for optimality results. Secondly, we analyze an extended version of Spectral Tetris and
prove that this algorithm indeed constructs optimally sparse frames \ahh{if performed after ordering the prescribed eigenvalues
blockwise}. Thus, Spectral Tetris can serve as an
algorithm for computing frames with this desirable property, and our results show that it is not possible to derive sparser frames
through a different procedure.

%**********************************************************************************************************************************
%**********************************************************************************************************************************

\subsection{Impact on Applications}

Frames are nowadays considered a fundamental tool in electrical engineering, and we wish to refer to the survey paper \cite{CK08}
as also to the introductory papers \cite{KC07a,KC07b}.
However, the application of frames for the analysis of high-dimensional data such a webpages labeled by over
a million parameters or databases of images, each image being one data point, typically suffers from the fact that frame measurements
are computationally not feasible due to constraints such as computing power and bandwidth, or even limited space to store
the synthesis matrix. With the results presented in this paper, we introduce {\em sparsity of frames} as a novel paradigm for
frame constructions, resulting in computationally highly efficient frames. Our results do not only provide a lower bound for the
maximally achievable sparsity, but with the Spectral Tetris algorithm explicit constructions of efficient frames
for high-dimensional data analysis are now possible. Certainly, the desire to construct {\em tight} frames is evident
due to the favorable reconstruction properties of such frames. But our results go beyond this case, and also enable constructions of
optimally sparse {\em non-tight} frames with prescribed eigenvalues of the frame operator.
Let us provide two \ahh{additional} exemplary
applications illustrating why such frame properties are a natural constraint and which areas our results are
anticipated to impact.

{\em Analysis of Streaming Signals.}
The structure of the frame operator plays a key role in the noise
rejection ability of the frame. When the frame coefficients are corrupted by additive white Gaussian noise, the mean-squared
error (MSE) in reconstructing the signal is minimized by choosing the frame to be tight. In the presence of colored noise,
however, a tight frame is no longer optimal and the frame operator needs to be matched to the noise covariance matrix.
In such cases, the frame needs to be designed with respect to the eigen-basis of the inverse noise covariance matrix and
its eigenvalues. This is similar in spirit to the water-filling principle for precoder design in wireless communication,
where transmit power is distributed across the eigen directions of an inverse channel-noise covariance matrix to equalize
signal-to-noise-ratio across eigen directions, see \cite{SSBGS02,PCL03}. Additional structure on the frame is typically required depending on the
application. When the signal to be decomposed is a time or space series that cannot be observed or processed over long
blocks---due to limited memory, aperture size, or computational power---then it is needed to have a frame that not only
has a prescribed operator but also requires access to the signal samples only over a small temporal or spatial window.
This motivates construction of frames with sparse elements and desired spectra. The constructions we develop in this
paper yield to 2-sparse frames in any dimension, where the frame coefficients for a signal in an $N$-dimensional space
can be computed by observing the data stream through a window of only two samples. The frame sparsity can be tailored
to any arbitrary basis. In particular, when the noise covariance is known, the frame can be made sparse with respect to
the eigen-basis of the inverse noise covariance matrix.

{\em Face Recognition.} In face recognition, one main objective is to classify faces according to some given criterion, for instance, to
distinguish male from female faces. The application of PCA (or similar algorithms) delivers a basis of eigenfaces.
Learning algorithms on some training set of faces can then, for each basis element, determine the degree of the
significance of its coefficients for determining the gender. Customarily, measurements taken to classify faces are
assumed to be affected by noise; hence frame expansions are desirable. The frame should ideally be designed to
match the degree of significance ($=$ eigenvalues) of the given basis in the sense that it should be more redundant
for the computation of the significant coefficients and less for the insignificant coefficients. This is precisely
the setting we consider in this paper, and for which we analyze and construct optimally sparse
frames.

These are just brief samples of applications which will benefit from the results and constructions developed in this
paper. We anticipate that also various other applications
are impacted.

%**********************************************************************************************************************************
%**********************************************************************************************************************************

\subsection{Outline}

This paper is organized as follows. In Section \ref{sec:frameconstruction}, we first fix the terminology we require from frame
theory and then review the extended version of the Spectral Tetris algorithm. A novel sparsity measure for a frame will then
be introduced in Section \ref{sec:sparsity} together with a notion of optimality. In Section \ref{sec:optimality result}, a structural
property of frames suitable for our analysis is first introduced, and finally we state and prove our main result Theorem \ref{theo:main}.
We finish with some conclusions and discussions in Section \ref{sec:conclusions}.

%**********************************************************************************************************************************
%**********************************************************************************************************************************
%**********************************************************************************************************************************

\section{Frame Construction}
\label{sec:frameconstruction}

We first review the \gk{initial} as well as the extended version of the Spectral Tetris algorithm from \cite{CCHKP10}. To stand on common ground,
we start by fixing our terminology while briefly reviewing the basic definitions and notations related to frames.

%**********************************************************************************************************************************
%**********************************************************************************************************************************

\subsection{Frames}

A sequence
$\Phi = (\varphi_i)_{i=1}^N$ in $\RR^n$ is called a {\em frame} for $\RR^n$, if it is a -- typically, but not necessarily
linearly dependent -- spanning set. This definition is equivalent to asking for the existence of constants $0 < A \le B < \infty$
such that
\[
A\norm{x}^2 \leq \sum_{i=1}^N |\langle x, \varphi_i \rangle |^2 \leq B\norm{x}^2
\quad \mbox{for all } x \in \RR^n.
\]
When $A$ is chosen as the largest possible value and $B$ as the smallest for these inequalities to hold, then we call them the {\em (optimal)
frame bounds}. If $A$ and $B$ can be chosen as $A=B$, then the frame $\Phi$ is called {\em $A$-tight}, and if $A=B=1$ is possible, $\Phi$ is a
{\em Parseval frame}. $\Phi$ is called {\em equal-norm}, if there exists some $c>0$ such that $\|\varphi_i\|=c$ for all $i=1,\ldots,N$, and
it is {\em unit-norm} if $c=1$.

Frames allow the analysis of data by studying the associated {\em frame coefficients} $(\langle x, \varphi_i \rangle)_{i=1}^N$, where
the operator $T$ defined by $T: \RR^n \to \ell_2(\{1, 2, \dots, N\})$, $x \mapsto (\langle x,\varphi_i\rangle)_{i=1}^N$
is called the \emph{analysis operator}. The adjoint $T^*$ of the analysis operator is typically
referred to as the {\em synthesis operator} and satisfies $T^*((c_i)_{i=1}^N) = \sum_{i=1}^N c_i\varphi_i$. Later, the
synthesis operator will play an essential role, and we will write it in the matrix form $[\varphi_1| \ldots | \varphi_N]$ with
the frame vectors as columns. In the sequel we refer to this matrix as the {\em synthesis matrix}.
The main operator associated with a frame, which provides a stable reconstruction process, is the {\em frame operator}
\[
S =T^* T : \RR^n \to \RR^n, \quad x \mapsto \sum_{i=1}^N \langle x,\varphi_i\rangle \varphi_i,
\]
a positive, self-adjoint, invertible operator on $\RR^n$. In the case of an $A$-tight frame, we have $S= A \cdot \Id_{\RR^n}$, and in case of
a Parseval frame, $S=\Id_{\RR^n}$. In general, $S$ allows for the reconstruction of a signal $x \in \RR^n$ through the
reconstruction formula
\beq \label{eq:expansion}
x = \sum_{i=1}^N \langle x,S^{-1} \varphi_i\rangle \varphi_i.
\eeq

Redundancy is obviously the crucial property of a frame ensuring resilience to noise and erasures while simultaneously enabling us
to choose the expansion coefficients appropriately. The particular choice of coefficients displayed in \eqref{eq:expansion} is the
smallest in $\ell_2$ norm \cite{Chr03}, hence it contains the least energy. Recently, a different view point has received rapidly
increasing attention, namely to choose the coefficient sequence to be sparse in the sense of having only few non-zero entries,
thereby allowing data compression while preserving perfect recoverability (see, e.g., \cite{BDE09}, and the references therein).
In this context, for later use, we will denote the support of a vector $x \in \RR^n$, i.e., the \gk{set of indices of the} non-zero entries, by $\supp x$.

Finally, we should mention that, customarily, redundancy of a frame $(\varphi_i)_{i=1}^N $ for $\RR^n$ was measured by $\frac{N}{n}$,
i.e., the number of frame vectors divided by the dimension of the ambient space. Since this measure is exceptionally crude and not
sensitive to local behavior of the frame vectors, the notions of {\em upper} and {\em lower redundancy} have been suggested in \cite{BCK10}
as a finer redundancy measure.

%**********************************************************************************************************************************
%**********************************************************************************************************************************

\subsection{The Spectral Tetris Algorithm}

Spectral Tetris was first introduced in \cite{CFMWZ09} as an algorithm to generate unit norm tight frames
for any number of frame vectors $N$, say, and for any ambient dimension $n$ provided that $\frac{N}{n} \ge 2$. This
algorithm \gk{is indeed significant for} frame constructions, since it is the first systematic construction
of unit norm tight frames. Before, only a number of very special classes of unit norm tight frames such as harmonic frames
have been known.

An extension to the construction of unit norm frames having a desired frame operator associated with eigenvalues
$\lambda_1, \ldots, \lambda_n \ge 2$ satisfying $\sum_{j=1}^n \lambda_j = N$ was then introduced and analyzed in
\cite{CCHKP10} -- in fact, an even more general algorithm for the construction of fusion frames was stated therein. The
frame-version of this algorithm is what we intend to analyze in this paper. Figure \ref{fig:ST} states the steps of
this version of the algorithm, which we coin {\em Spectral Tetris for Frames}; in short, STF. We wish to remark that
the original form of the algorithm in \cite{CFMWZ09} requires the sequence of eigenvalues to be in decreasing order,
i.e. $\lambda_1 \ge \ldots \ge \lambda_n$. This assumption,
however, was made only for classification reasons, and it is easily seen that it can be dropped. Since in the sequel,
we will consider carefully chosen, presumably non-decreasing, sequences of eigenvalues, the gained freedom is essential
for our analysis.
%Nevertheless all eigenvalues have to be $\geq 2$.

\begin{figure}[h]
\centering
\framebox{
\begin{minipage}[h]{5.0in}
\vspace*{0.3cm}
{\sc \underline{STF: Spectral Tetris for Frames}}

\vspace*{0.4cm}

{\bf Parameters:}
\begin{itemize}
\item Dimension $n \in \NN$.
\item Number of frame elements $N \in \NN$.
\item Sequence of eigenvalues $\lambda_1, \ldots, \lambda_n \ge 2$ satisfying $\sum_{j=1}^n \lambda_j = N$.
\end{itemize}

\vspace*{0.2cm}

{\bf Algorithm:}
\begin{itemize}
\item[1)] Set $i := 1$.
\item[2)] For $j=1,\ldots,n$ do
\item[3)] \hspace*{1cm} Repeat
\item[4)] \hspace*{2cm} If $\lambda_j < 1$ then
\item[5)] \hspace*{3cm} $\varphi_i := \sqrt{\frac{\lambda_j}{2}} \cdot e_j + \sqrt{1-\frac{\lambda_j}{2}} \cdot e_{j+1}$.
\item[6)] \hspace*{3cm} $\varphi_{i+1} := \sqrt{\frac{\lambda_j}{2}}\cdot e_j - \sqrt{1-\frac{\lambda_j}{2}} \cdot e_{j+1}$.
\item[7)] \hspace*{3cm} $i := i+2$.
\item[8)] \hspace*{3cm} $\lambda_{j+1} := \lambda_{j+1} - (2-\lambda_j)$.
\item[9)] \hspace*{3cm} $\lambda_j := 0$.
\item[10)] \hspace*{2cm} else
\item[11)] \hspace*{3cm} $\varphi_i := e_j$.
\item[12)] \hspace*{3cm} $i := i+1$.
\item[13)] \hspace*{3cm} $\lambda_j := \lambda_j - 1$.
\item[14)] \hspace*{2cm} end.
\item[15)] \hspace*{1cm} until $\lambda_j = 0$.
\item[16)] end.
\end{itemize}

\vspace*{0.2cm}

{\bf Output:}
\begin{itemize}
\item Frame STF$(N; \lambda_1, \ldots, \lambda_n) :=\{\varphi_i\}_{i=1}^N$.
\end{itemize}
\vspace*{0.01cm}
\end{minipage}
}
\caption{The Spectral Tetris algorithm for constructing an $N$-element unit norm frame STF$(N; \lambda_1, \ldots, \lambda_n)$ for $\RR^n$
with an associated frame operator having eigenvalues $\lambda_1, \ldots, \lambda_n$.}
\label{fig:ST}
\end{figure}

Before we continue, let us give an example to provide a more intuitive feeling of how STF works and to introduce a cursor notation which
will be utilized in later proofs.

\begin{example} \label{exa:cursor}
We aim to construct a $\ahh{10}$ element unit norm frame in $\mathbb{R}^4$ having the eigenvalues $\lambda_1=\lambda_2=\lambda_3=\frac{8}{3}$ and
$\lambda_4=2$. STF will provide such a frame by generating a $\ahh{4 \times 10}$ synthesis matrix with the following properties: The columns of
the matrix have norm $1$ (guaranteeing that the frame has unit norm vectors) and the rows of the matrix are orthogonal and square sum to
the desired eigenvalues  (guaranteeing that the frame operator is diagonal with the desired eigenvalues on its diagonal). In the example
we consider each of the first $3$ rows of the to-be-generated synthesis matrix have to square sum to $\frac{8}{3}$ and the last row to
square sum to $2$. The algorithm now starts with a $\ahh{4 \times 10}$ matrix of unknown entries and lets a {\em cursor} move forward along columns
and rows assigning values to certain entries. The remaining entries are set to zero in the end. When the cursor is in position $(i,j)$, we
update the variable $\lambda_i$ to be the difference between the eigenvalue assigned to row $i$ and the square sum of the entries already
assigned to row $i$, in order to keep track of how much weight still has to be assigned to row $i$ to make it square sum to the desired
eigenvalue. In general, one of the three cases occurs in each step:

{\em Case $1$}: If $\lambda_i>1$, then the current entry $(i,j)$ is set to one, we update $\lambda_i := \lambda_i-1$, and the cursor $(i,j)$
is moved to the right, i.e., $(i,j) := (i+1,j)$. This is, for example, the case when the cursor is in position $(1,1)$.
\ahh{At this point we have $\lambda_1=\frac{8}{3}$ and we update to $\lambda_1=\frac{8}{3}-1=\frac{5}{3}$.
The matrix changes as follows, where we denote the unknown matrix entries by $\cdot$ and the position of the cursor by $\odot$:}
\[
\begin{bmatrix}
\ahh{\odot}&\cdot&\cdot&\cdot&\cdot&\cdot&\cdot&\cdot&\cdot&\cdot\\
\cdot&\cdot&\cdot&\cdot&\cdot&\cdot&\cdot&\cdot&\cdot&\cdot\\
\cdot&\cdot&\cdot&\cdot&\cdot&\cdot&\cdot&\cdot&\cdot&\cdot\\
\cdot&\cdot&\cdot&\cdot&\cdot&\cdot&\cdot&\cdot&\cdot&\cdot
\end{bmatrix}
\quad \longrightarrow \quad
\begin{bmatrix}
1&\ahh{\odot}&\cdot&\cdot&\cdot&\cdot&\cdot&\cdot&\cdot&\cdot\\
\cdot&\cdot&\cdot&\cdot&\cdot&\cdot&\cdot&\cdot&\cdot&\cdot\\
\cdot&\cdot&\cdot&\cdot&\cdot&\cdot&\cdot&\cdot&\cdot&\cdot\\
\cdot&\cdot&\cdot&\cdot&\cdot&\cdot&\cdot&\cdot&\cdot&\cdot
\end{bmatrix}
\]

{\em Case $2$}: If $0< \lambda_i<1$, then the entries $(i,j)$, $(i+1,j)$, $(i,j+1)$, and $(i+1,j+1)$ are set according to lines $5)$ and $6)$ of
STF, we update $\lambda_{i+1} := \lambda_{i+1} + \lambda_i - 2$, and the cursor is moved to $(i,j) := (i+2,j+1)$.
This is, for example, the case when the cursor is in position $(1,3)$.
\ahh{At this point we have $\lambda_1 = \frac{2}{3}$ and update $\lambda_2 = \frac{8}{3}$ to $\lambda_2 = \frac{8}{3}- (2-\frac{2}{3})=\frac{4}{3}$. The matrix changes as follows:}

\[
\begin{bmatrix}
1&1&\ahh{\odot}&\cdot&\cdot&\cdot&\cdot&\cdot&\cdot&\cdot\\
\cdot&\cdot&\cdot&\cdot&\cdot&\cdot&\cdot&\cdot&\cdot&\cdot\\
\cdot&\cdot&\cdot&\cdot&\cdot&\cdot&\cdot&\cdot&\cdot&\cdot\\
\cdot&\cdot&\cdot&\cdot&\cdot&\cdot&\cdot&\cdot&\cdot&\cdot
\end{bmatrix}
\quad \longrightarrow \quad
\begin{bmatrix}
1&1&\sqrt{\frac{1}{3}}&\sqrt{\frac{1}{3}}&\cdot&\cdot&\cdot&\cdot&\cdot&\cdot\\
\cdot&\cdot&\sqrt{\frac{2}{3}}&-\sqrt{\frac{2}{3}}&\ahh{\odot}&\cdot&\cdot&\cdot&\cdot&\cdot\\
\cdot&\cdot&\cdot&\cdot&\cdot&\cdot&\cdot&\cdot&\cdot&\cdot\\
\cdot&\cdot&\cdot&\cdot&\cdot&\cdot&\cdot&\cdot&\cdot&\cdot
\end{bmatrix}
\]
This case is the crucial step of the algorithm. Note that the $2\times2$ block we inserted has the properties that its rows are orthogonal and
its columns square sum to $1$, which are properties desired for the synthesis operator.

{\em Case $3$}: If $\lambda_i = 1$, then the entry $(i,j)$ is set to one, and the cursor is moved to the right below $(i,j) := (i+1,j+1)$.
This is, for example, the case when the cursor is $(3,8)$. \ahh{At this point we have $\lambda_3=1$ and the matrix changes as follows:}
\[
\begin{bmatrix}
1&1&\sqrt{\frac{1}{3}}&\sqrt{\frac{1}{3}}&\cdot&\cdot&\cdot&\cdot&\cdot&\cdot\\
\cdot&\cdot&\sqrt{\frac{2}{3}}&-\sqrt{\frac{2}{3}}&1&\sqrt{\frac{1}{6}}&\sqrt{\frac{1}{6}}&\cdot&\cdot&\cdot\\
\cdot&\cdot&\cdot&\cdot&\cdot&\sqrt{\frac{5}{6}}&-\sqrt{\frac{5}{6}}&\ahh{\odot}&\cdot&\cdot\\
\cdot&\cdot&\cdot&\cdot&\cdot&\cdot&\cdot&\cdot&\cdot&\cdot
\end{bmatrix}
\quad \longrightarrow \quad
\begin{bmatrix}
1&1&\sqrt{\frac{1}{3}}&\sqrt{\frac{1}{3}}&\cdot&\cdot&\cdot&\cdot&\cdot&\cdot\\
\cdot&\cdot&\sqrt{\frac{2}{3}}&-\sqrt{\frac{2}{3}}&1&\sqrt{\frac{1}{6}}&\sqrt{\frac{1}{6}}&\cdot&\cdot&\cdot\\
\cdot&\cdot&\cdot&\cdot&\cdot&\sqrt{\frac{5}{6}}&-\sqrt{\frac{5}{6}}&1&\cdot&\cdot\\
\cdot&\cdot&\cdot&\cdot&\cdot&\cdot&\cdot&\cdot&\ahh{\odot}&\cdot
\end{bmatrix}.
\]

After performing all steps of the algorithm, the final synthesis matrix constructed by STF has the form
\[
\begin{bmatrix}
1&1&\sqrt{\frac{1}{3}}&\sqrt{\frac{1}{3}}&0&0&0&0&0&0\\
0&0&\sqrt{\frac{2}{3}}&-\sqrt{\frac{2}{3}}&1&\sqrt{\frac{1}{6}}&\sqrt{\frac{1}{6}}&0&0&0\\
0&0&0&0&0&\sqrt{\frac{5}{6}}&-\sqrt{\frac{5}{6}}&1&0&0\\
0&0&0&0&0&0&0&0&1&1
\end{bmatrix} .
\]
\end{example}

%**********************************************************************************************************************************
%**********************************************************************************************************************************
%**********************************************************************************************************************************

\section{New Paradigm for Frame Constructions: Sparsity}
\label{sec:sparsity}

%**********************************************************************************************************************************
%**********************************************************************************************************************************

\subsection{Classical Sparsity}

Over the past few years, sparsity has become a key concept in various areas of applied mathematics, computer science, and
electrical engineering. Sparse signal processing methodologies explore the fundamental fact that many types of signals can
be represented by only a few non-zero coefficients when choosing a suitable basis or, more generally, a frame. A signal
representable by only $k$, say, basis or frame elements is called {\em $k$-sparse}. If signals possess such a
sparse representation, they can in general be recovered from few measurements using $\ell_1$ minimization techniques (see, e.g.,
\cite{BDE09,CRT06,Don06} and the references therein).

%**********************************************************************************************************************************
%**********************************************************************************************************************************

\subsection{Sparse Frames}

In this paper, however, we pose a different question concerning sparsity, viewing sparsity from a very different standpoint.
Typically, data processing applications face low on-board computing power and/or a small bandwidth budget. When the signal
dimension is large, the decomposition of the signal into its frame measurements requires a large number of additions and
multiplications, which may be infeasible for on-board data processing. Also the space required for storing the synthesis matrix of
the frame might be huge.
It would hence be a significant improvement, if each frame vector would
contain very few non-zero entries, hence -- phrasing it differently -- be sparse in the standard unit vector basis, which ensures low-complexity
processing. Since we are interested in the performance of the whole frame, the total number of non-zero entries in the frame
vectors seems to be a suitable sparsity measure. This viewpoint can also be slightly generalized by assuming that there exists
a unitary transformation mapping the frame into one having this `sparsity' property.

%**********************************************************************************************************************************
%**********************************************************************************************************************************

\subsection{Sparseness Measure}

Taking these considerations into account, we are led to proclaim the following definition for a sparse frame:

\begin{definition} \label{def:k_sparse}
Let $(e_j)_{j=1}^n$ be an orthonormal basis for $\RR^n$. Then a frame $(\varphi_i)_{i=1}^N$ for $\RR^n$ is called
{\em $k$-sparse} with respect to $(e_j)_{j=1}^n$, if, for each $i \in \{1,\ldots,N\}$, there exists $J_i \subseteq \{1,\ldots,n\}$
such that
\[
\varphi_i \in \spann\{e_j : j \in J_i\}
\]
and
\beq \label{eq:k_sparse_1}
\sum_{i=1}^n |J_i| = k.
\eeq
\end{definition}

The attentive reader will have realized that this definition differs from the definition stated in \cite{CCHKP10} for fusion
frames (see \cite{CKL08}) when restricting to the special case of frames. The exact relation is the following:
Let $(e_j)_{j=1}^n$ be an orthonormal basis for $\RR^n$, let $(\varphi_i)_{i=1}^N$ be a frame for $\RR^n$, and, for each
$i \in \{1,\ldots,N\}$, let $J_i \subseteq \{1,\ldots,n\}$ such that $\varphi_i \in \spann\{e_j : j \in J_i\}$. Then,
in the sense of \cite{CCHKP10}, the frame is $\max\{|J_i| : i = 1, \ldots, n\}$-sparse, whereas in our Definition
\ref{def:k_sparse}, the frame is $\sum_{i=1}^n |J_i|$ sparse. Thus our definition encodes the true overall sparsity which is
the sparsity required for frame processing in contrast to the more local version of \cite{CCHKP10}.

One can certainly imagine other sparsity measures dependent on the requirements and constraints of the application at hand.
Instead of \eqref{eq:k_sparse_1}, a weighted version could be considered with the weights chosen depending on the computational
constraints of the application. Also, \eqref{eq:k_sparse_1} could be regarded as the $\ell_1$ norm of the sequence $\{|J_i| :
i = 1, \ldots, n\}$, and a different viewpoint might lead us to considering a different norm instead -- as it was done in
\cite{CCHKP10} for the $\ell_\infty$ norm.

%**********************************************************************************************************************************
%**********************************************************************************************************************************

\subsection{Notion of Optimality}

We now have the necessary machinery at hand to introduce a notion of an {\em optimally} sparse frame. Optimality will typically
-- as also in this paper - be considered within a particular class of frames, for instance, in the class of unit norm tight frames.

\begin{definition}
Let $\cF$ be a class of frames for $\RR^n$, let $(\varphi_i)_{i=1}^N \in \cF$, and let $(e_j)_{j=1}^n$ be an orthonormal basis for $\RR^n$.
Then $(\varphi_i)_{i=1}^N$ is called {\em  optimally sparse in $\cF$ with respect to $(e_j)_{j=1}^n$}, if $(\varphi_i)_{i=1}^N$ is
$k_1$-sparse with respect to $(e_j)_{j=1}^n$ and there does not exist a frame $(\psi_i)_{i=1}^N \in \cF$ which is $k_2$-sparse
with respect to $(e_j)_{j=1}^n$ with $k_2 < k_1$.
\end{definition}

We wish to emphasize the strong dependence of sparsity on the chosen basis. Also an optimally sparse frame is in general not uniquely determined;
we present an example for this observation in Subsection \ref{subsec:maxsparsity}.

%**********************************************************************************************************************************
%**********************************************************************************************************************************
%**********************************************************************************************************************************

\section{An Optimality Result for Sparse Frames}
\label{sec:optimality result}

We now seek a construction for an optimally sparse unit norm frame with prescribed properties. As already elaborated upon before, the
condition we impose is having a given frame operator, which, in particular, also includes operators with equal eigenvalues corresponding
to tight frames. This frame operator will in the following be always determined by its eigenvalues. Hence we are interested in optimal
sparsity within the following class:

Let $n, N > 0$ and let the real values $\lambda_1,\ldots,\lambda_n\geq 2$ satisfy $\sum_{j=1}^n \lambda_j = N$. Then
the class of unit norm frames $(\varphi_i)_{i=1}^N$ in $\RR^n$ whose frame operator has eigenvalues $\lambda_1, \ldots, \lambda_n$ will
be denoted by
\[
\cF(N,\{\lambda_i\}_{i=1}^{n}).
\]
It is important to mention that by writing $\{\lambda_i\}_{i=1}^{n}$, we wish to indicate that the ordering does not play a role
here, however, multiplicities are counted. The just defined class $\cF(N,\{\lambda_i\}_{i=1}^{n})$
\ahh{is non-empty by application of the STF. In fact, it can be shown using methods introduced in \cite{CK03} and \cite{Cas04},
that it is an infinite set for any $n, N > 0$ and real values $\lambda_1,\ldots,\lambda_n\geq 2$.}

It might be beneficial for the reader to mention at this point that we will discuss the analysis presented in Subsections
\ref{subsec:structure} to \ref{subsec:main} in the important special case of tight frames in Subsection \ref{subsec:tight} for illustrative
purposes.

%**********************************************************************************************************************************
%**********************************************************************************************************************************

\subsection{Novel Structural Property of Synthesis Matrices}
\label{subsec:structure}

Aiming for determining the \fk{maximally} achievable sparsity for a class $\cF(N,\{\lambda_i\}_{i=1}^{n})$, we first need
to introduce a particular measure associated with the set of eigenvalues $\{\lambda_i\}_{i=1}^{n}$. This measure indicates the
maximal number of partial sums which are an integer; here one maximizes over all reorderings of the eigenvalues. Before
stating the precise definition, let us provide some intuition why the maximally achievable sparsity is dependent on partial
integer valued sums. Whenever the partial sum of the first $l$, say, eigenvalues
$\lambda_1,\ldots,\lambda_l$ is an integer, the cursor -- recall that the concept of a cursor was introduced in Example \ref{exa:cursor}
-- in row $l$ will be in Case $3$ of the cases discussed in Example \ref{exa:cursor}. Opposed to the setting of 4 new non-zeros entries
in Case 2, in this case only {\em one} new non-zero entry will be defined. Roughly speaking, this will allow us to reduce the analysis
to the blocks between two such integer partial sums.

%The maximally achievable sparsity will depend on this number, since an ordering of the eigenvalues has to be
%fixed before performing STF.

The precise definition of the measure on a set of eigenvalues  we require is now as follows.

\begin{definition}
A finite sequence of real values $\lambda_1,\ldots,\lambda_n$ is {\em ordered blockwise}, if for any permutation $\pi$ of $\{1,\ldots,n\}$
the set of partial sums $\{\sum_{j=1}^s\lambda_j\colon s=1,\ldots,n\}$ contains at least as many integers as the set
$\{\sum_{j=1}^s\lambda_{\pi(j)}\colon s=1,\ldots,n\}$.
The {\em maximal block number} of a finite sequence of real values $\lambda_1, \ldots, \lambda_n$, denoted by
$\mu(\lambda_1, \ldots, \lambda_n)$, is the number of integers in $\{\sum_{j=1}^s\lambda_{\sigma(j)}\colon s=1,\ldots,n\}$, where
$\sigma$ is a permutation of $\{1,\ldots,n\}$ such that $\lambda_{\sigma(1)},\ldots,\lambda_{\sigma(n)}$ is ordered blockwise.
\end{definition}

Surprisingly, the notion of maximal block number can illuminatingly be transferred to a particular decomposition
property of the synthesis matrix of a frame. Let us first define the decomposition property we are interested in:

\begin{definition}
Let $n, N > 0$, and let $(\varphi_i)_{i=1}^N$ be a frame for $\RR^n$. Then we say that the synthesis matrix of  $(\varphi_i)_{i=1}^N$
has {\em block decomposition of order $m$}, if there exists a partition $\{1, \ldots, N\} = I_1 \cup \ldots \cup I_m$ such that,
for any $k_1 \in I_{i_1}$ and $k_2 \in I_{i_2}$ with $i_1 \neq i_2$, we have $\supp \varphi_{k_1} \cap \supp \varphi_{k_2} = \emptyset$
and $m$ is maximal.
\end{definition}

The following result now connects the maximal block number of the sequence of eigenvalues of a frame operator
with the block decomposition order of an associated frame.

\begin{proposition}\label{prop:block}
Let $n, N > 0$ and let the real values $\lambda_1,\ldots,\lambda_n \ge 2$ satisfy $\sum_{j=1}^n \lambda_j = N$.
Then the synthesis matrix of any frame in the class $\cF(N,\{\lambda_i\}_{i=1}^{n})$ has block decomposition
of order at most $\mu(\lambda_1, \ldots, \lambda_n)$.
\end{proposition}

\begin{proof}
Suppose $(\varphi_i)_{i=1}^N\in\cF(N,\{\lambda_i\}_{i=1}^{n})$ has block decomposition of order
$\nu$ and let
$\{1, \ldots, N\} = I_1 \cup \ldots \cup I_\nu$  be a corresponding partition.
For $j=1,\ldots,\nu$, let $S_j$ be the common support set of the vectors $(\varphi_i)_{i\in I_j}$, i.e., $k\in S_j$ if and only if
$k\in\supp\varphi_i$ for some $i\in I_j$. Now let $r_k$ denote the $k$-th row of the synthesis matrix of $(\varphi_i)_{i=1}^N$.
Then $S_1,\ldots,S_{\nu}$ is a partition of $\{1,\ldots,n\}$
and, for every $j=1,\ldots,\nu$ \ah{we have by the fact that $(\varphi_i)_{i=1}^N$ consists of unit norm vectors and by our choice of
$I_j$ and $S_j$ that}
\begin{equation}{\label{reviewer3}}
|I_j|= \sum_{k\in I_j} \|\varphi_k\|^2 =\sum_{k\in S_j} \|r_k\|^2 = \sum_{k\in S_j} \lambda_k.
\end{equation}
\ah{The last equality holds since we, after permutation of the columns, can write the synthesis matrix of $(\varphi_i)_{i=1}^N$ as
$T^* = [T^*_1,\ldots,T^*_{\nu}]$, where $T^*_j$ has zero entries except on the rows indexed by $I_j$ and the columns indexed by
$S_j$, for $j=1,\ldots,\nu$. The frame operator $T^*T = \sum_{j=1}^{\nu} T^*_j T_j$ is block diagonal with blocks $T^*_j T_j$, hence
its eigenvalues are exactly the union of those of each matrix $T^*_j T_j$. But $\sum_{k\in I_j} \|\varphi_k\|^2 =\sum_{k\in S_j} \|r_k\|^2 $
equals the} \fk{square of the} \ah{Hilbert-Schmidt norm of $T^*_j$ and therefore the sum of the eigenvalues of $T^*_jT_j$. This shows the
last equality of (\ref{reviewer3}). Since (\ref{reviewer3}) holds for all $j=1,\ldots,\nu$, we conclude that} the maximal block number of
$\lambda_1, \ldots, \lambda_n$ is at least $\nu$.
Thus the synthesis matrix of the arbitrarily chosen frame $(\varphi_i)_{i=1}^N$ in the class $\cF(N,\{\lambda_i\}_{i=1}^{n})$ has
block decomposition of order at most $\mu(\lambda_1, \ldots, \lambda_n)$.
\end{proof}

%**********************************************************************************************************************************
%**********************************************************************************************************************************

\subsection{\fk{Maximally} Achievable Sparsity}
\label{subsec:maxsparsity}

Having introduced the required new notions, we are now in a position to state the exact value for the maximally achievable
sparsity for a class $\cF(N,\{\lambda_i\}_{i=1}^{n})$. It is not initially clear that this optimal sparsity can always be attained.
With Theorem \ref{theo:main} we will prove that this is indeed the case; in fact, Theorem \ref{theo:main} also provides an explicit
construction of those frames.

\begin{theorem} \label{theo:maxsparsity}
Let $n, N > 0$, and let the real values $\lambda_1,\ldots,\lambda_n \ge 2$ satisfy $\sum_{j=1}^n \lambda_j = N$. Then any frame
in $\cF(N,\{\lambda_i\}_{i=1}^{n})$ has sparsity at least
\[
N+2(n-\mu(\lambda_1, \ldots, \lambda_n))
\]
with respect to any orthonormal basis.
\end{theorem}

\begin{proof}
We first study the case that $\mu(\lambda_1, \ldots, \lambda_n)=1$. For this, let $T^*$ denote the synthesis matrix of a frame in
$\cF(N,\{\lambda_i\}_{i=1}^{n})$ with respect to a fixed orthonormal basis. For the sake of brevity, in the sequel we will use
the phrase that two rows of $T^*$ {\em have overlap of size} $k$, if the intersection of their supports is a set of size $k$. Note that,
since the rows of $T^*$ are orthogonal, it is not possible that two rows of $T^*$ have overlap $1$.

Fix now an arbitrary row $r_1$ of $T^*$. Since, by Proposition \ref{prop:block}, $T^*$ has block decomposition of order $1$, there
exists a row $r_2$ whose overlap with $r_1$ is of size $\geq 2$. Similarly, there has to exist a row different from $r_1$ and $r_2$
which has overlap of size $\geq 2$ with $r_1$ or $ r_2$. Iterating this procedure will provide an order $r_1, r_2,\dots r_n$ such
that, for each row $r_j$, there exists some $k<j$ such that $r_j$ has overlap of size $\geq 2$ with $r_k$. Since all columns in $T^*$
are unit norm, for each column $c$, there exists a minimal $j$ for which the entry $c_{r_j}$ is non-zero. This yields $N$
non-zero entries in $T^*$. In addition, each row $r_2$ through $r_n$ has at least 2 non-zero entries coming from the overlap,
which are different from the just accounted for $N$ entries, \ah{since these entries \fk{cannot} be the
non-zero entries of minimal index of a column due to the overlap with a previous row.}
This sums up to a total of at least $2(n-1)$ non-zero coefficients.
Consequently, the synthesis matrix has at least $N+2(n-1)$ non-zero entries, as desired.

Finally, suppose $\mu:=\mu(\lambda_1, \ldots, \lambda_n)>1$. By Proposition~\ref{prop:block}, $T^*$ has block decomposition of order
at most $\mu$. Performing the same construction as above, there exist at most $\mu$ rows $r_j$ (including the first one) which do not
have overlap with a row $r_k$, $k<j$. Thus the synthesis matrix $T^*$ must at least contain $N+2(n-\mu)$ non-zero entries.
\end{proof}

It should be mentioned that an optimally sparse frame from $\cF(N,\{\lambda_i\}_{i=1}^{n})$ is in general not uniquely determined.
Various examples can be constructed along the following line: For simplicity, we choose $n=4$ and $N=9$ and construct a tight frame,
i.e., $\lambda_1=\ldots=\lambda_4=\frac{9}{4}$.
%. Notice that  so that the ordering of the eigenvalues becomes irrelevant.
Then, by Theorem \ref{theo:maxsparsity}, the maximally achievable sparsity is $9+2(4-1)=15$. The following matrices are synthesis matrices
with respect to the standard unit vector basis of two different frames in $\cF(9, \{\frac{9}{4}\}_{i=1}^9)$, the first in fact being generated by
Spectral Tetris:
\[
\begin{bmatrix}
1&1&\sqrt{\frac{1}{8}}&\sqrt{\frac{1}{8}}&0&0&0&0&0\\
0&0&\sqrt{\frac{7}{8}}&-\sqrt{\frac{7}{8}}&\sqrt{\frac{1}{4}}&\sqrt{\frac{1}{4}}&0&0&0\\
0&0&0&0&\sqrt{\frac{3}{4}}&-\sqrt{\frac{3}{4}}&\sqrt{\frac{3}{8}}&\sqrt{\frac{3}{8}}&0\\
0&0&0&0&0&0&\sqrt{\frac{5}{8}}&-\sqrt{\frac{5}{8}}&1\\
\end{bmatrix}
\]
and
\[
\begin{bmatrix}
1&\sqrt{\frac{5}{8}}&\sqrt{\frac{5}{8}}&0&0&0&0&0&0\\
0&\sqrt{\frac{3}{8}}&-\sqrt{\frac{3}{8}}&\sqrt{\frac{3}{8}}&\sqrt{\frac{3}{8}}&\sqrt{\frac{3}{8}}&\sqrt{\frac{3}{8}}&0&0\\
0&0&0&\sqrt{\frac{5}{8}}&-\sqrt{\frac{5}{8}}&0&0&1&0\\
0&0&0&0&0&\sqrt{\frac{5}{8}}&-\sqrt{\frac{5}{8}}&0&1\\
\end{bmatrix}.
\]

%**********************************************************************************************************************************
%**********************************************************************************************************************************

\subsection{Main Result}
\label{subsec:main}

Having set the benchmark, we now prove that frames constructed by Spectral Tetris in fact achieve the optimal
sparsity rate. For this, we would like to remind the reader that the frame constructed by Spectral Tetris
(see Figure \ref{fig:ST}) was denoted by STF$(N; \lambda_1, \ldots, \lambda_n)$.

\begin{theorem} \label{theo:main}
Let $n, N > 0$, and let the real values $\lambda_1,\ldots,\lambda_n \ge 2$ be ordered blockwise and satisfy $\sum_{j=1}^n \lambda_j = N$.
Then the frame STF$(N; \lambda_1, \ldots, \lambda_n)$ is optimally sparse in $\cF(N,\{\lambda_i\}_{i=1}^{n})$ with respect to the standard unit vector basis.
That is, this frame is $N+2(n-\mu(\lambda_1, \ldots, \lambda_n))$-sparse with respect to the standard unit vector basis.
\end{theorem}

\begin{proof}
Let $(\varphi_i)_{i=1}^N$ be the frame STF$(N; \lambda_1,\ldots,\lambda_n)$. We will first show that its synthesis matrix
has block decomposition of order $\mu:=\mu(\lambda_1, \ldots, \lambda_n)$. For this, let $k_0=0$, and let $k_1,\ldots,k_{\mu}\in\mathbb{N}$
be chosen such that $\ah{m_i:=\sum_{j=1}^{k_i}\lambda_j}$ is an integer for every $i=1,\ldots,\mu$. Moreover, let $m_0=0$.
Further, note that $k_{\mu}=n$ and $m_{\mu}=N$, since $\sum_{j=1}^n \lambda_j$ is an integer by hypothesis. The steps of Spectral Tetris (STF)
for computing STF$(m_1; \lambda_1, \ldots, \lambda_{k_1})$ and STF$(N; \lambda_1, \ldots, \lambda_n)$ \ah{coincide until the cursor }\fk{index}\ah{ in
computing STF$(N; \lambda_1, \ldots, \lambda_n)$ reach}\fk{es}\ah{ $(k_1,m_1)$}. Therefore, the first $k_1$ entries of the first $m_1$
vectors of both constructions coincide. Continuing the computation of STF$(N; \lambda_1, \ldots, \lambda_n)$ will set the remaining
entries of the first $m_1$ vectors and also the first $k_1$ entries of the remaining vectors to zero. Thus, any of the first
$k_1$ vectors has disjoint support from any of the vectors constructed later on. Repeating this argument for $k_2$ until $k_{\mu}$,
we obtain that the synthesis matrix has a block decomposition of order $\mu$; the corresponding partition of the frame vectors being
\[
\bigcup_{i=1}^{\mu}\{\varphi_{m_{i-1}+1},\ldots,\ah{\varphi_{m_{i}}}\}.
\]

To compute the number of non-zero entries in the synthesis matrix generated by Spectral Tetris, we let $i\in\{1,\ldots,\mu\}$ be
arbitrarily fixed and compute the number of non-zero entries of the vectors $\varphi_{m_{i-1}+1},\ldots,\ah{\varphi_{m_{i}}}$.
Spectral Tetris ensures that each of the rows $k_{i-1}+1$ up to $k_{i}-1$ intersects the support of the subsequent row on a set
of size $2$, \ah{since in these rows }\fk{STF will always proceed as }\ah{in case $2$ of the three cases} \fk{in}\ah{ the spectral tetris example above.}
Thus, there exist $2(k_{i}-k_{i-1}-1)$ frame vectors with two non-zero entries. The remaining
$\ah{(m_i-m_{i-1})}-2(k_i-k_{i-1}-1)$
frame vectors will have only one entry, yielding a total number of $\ah{(m_i-m_{i-1})}+2(k_i-k_{i-1}-1)$ non-zero entries in the vectors
$\varphi_{m_{i-1}+1},\ldots,\ah{\varphi_{m_{i}}}$.

Summarizing, the total number of non-zero entries in the frame vectors of $(\varphi_i)_{i=1}^N$ is
\[
\sum_{i=1}^{\mu}\ah{(m_i-m_{i-1})}+2(k_i-k_{i-1}-1)
=\left(\sum_{i=1}^{\mu}\ah{(m_i-m_{i-1})}\right)+
2\left(k_{\mu}-\left(\sum_{i=1}^{\mu}1\right)\right)
=N+2(n-\mu),
\]
which is by Theorem \ref{theo:maxsparsity} the maximally achievable sparsity.
\end{proof}

The reader will have realized that Spectral Tetris generates frames which are `only' optimally sparse with
respect to the standard unit vector basis. This seems at first sight like a drawback. However, if sparsity with respect
to a different orthonormal basis is required, Spectral Tetris can easily be modified to accommodate this request by using
vectors of this orthonormal basis instead of the standard unit vector basis when filling in the frame vectors in Steps 5,
6 and 11 in STF (cf. Figure \ref{fig:ST}). It is a straightforward exercise to show that this modified Spectral Tetris algorithm
then generates a frame which is optimally sparse with respect to this new orthonormal basis.

%**********************************************************************************************************************************
%**********************************************************************************************************************************

\subsection{Special Case: Constructing Optimally Sparse Tight Frames}
\label{subsec:tight}

In the special case of equal eigenvalues, i.e., of tight frames, with $N$ elements in $\RR^n$, all eigenvalues need to
equal $\frac{N}{n}$ for the equality $\sum_{j=1}^n \lambda_j = N$ to be satisfied. The maximal block number
can be easily computed to be $\gcd(N,n)$. Theorem \ref{theo:main} then takes the following form:

\begin{corollary}
For $n, N > 0$, the frame STF$(N; \frac{N}{n}, \ldots, \frac{N}{n})$ is optimally sparse in
$\cF(N,\{\frac{N}{n}\}_{i=1}^N)$ with respect to
the standard unit vector basis.
That is, this frame is $N+2(n-\gcd(N,n))$-sparse with respect to the standard unit vector basis.
\end{corollary}

%**********************************************************************************
%**********************************************************************************
%**********************************************************************************

\section{Conclusions and Discussion}
\label{sec:conclusions}

In this paper we considered the design of frames which enable efficient computations of the associated frame measurements.
This led to the introduction of the notion of a sparse frame as well as a sparsity measure for such frames, thereby
introducing optimal sparsity as a new paradigm into the construction of frames. We then analyzed an extended version of
Spectral Tetris for frames and proved that the frames constructed by this algorithm are indeed optimally sparse.
This shows that Spectral Tetris can serve as an algorithm for computing frames with this desirable property, and
our results prove that it is not possible to derive sparser frames through a different procedure.

We would finally like to point out that the analysis in this paper leads to several intriguing open problems for
future research; a few examples are stated in the sequel.

\bitem
\item {\it Eigenvalues also smaller than $2$.} It is still an open problem whether and how Spectral Tetris extends
to sets of eigenvalues, if some eigenvalues are smaller than $2$. \gk{The extension of the} Spectral Tetris
\gk{algorithm} by inserting larger DFT matrices than the previously exploited $2 \times 2$-matrices, \gk{allowed }\fk{ for }\gk{
some partial results (see \cite{CFH10}). }
However, from the results in \cite{CFH10} it can be deduced that this procedure does not always lead to optimally
sparse frames even in the case when all eigenvalues are equal, i.e., the tight frame case. Hence, extensive research
will be necessary to introduce an appropriate -- in the sense of optimal sparsity -- extension of Spectral Tetris.

\item {\it Extension to other classes of frames.} Depending on the application, other desiderata might be requested
from a frame such as, for instance, equi-angularity. For such a class of frames, the question of an optimally
sparse frame can and should similarly be posed.

\item {\it Relative sparsity/Compressibility.} Taking numerical considerations and perturbations into account, it
will be necessary to extend the notion of sparsity to relative sparsity/compressibility for frames and analyze
optimality for such.

%Since the eigenvalues might only be available in an approximate form, the
%question arises whether the optimal sparsity of the class of frames is robust under perturbations of the eigenvalues.
%This will require the introduction of a notion of relative sparsity/compressibility instead of sparsity for frames.
%Also in this situation Spectral Tetris might not be directly applicable

\eitem

%**********************************************************************************
%**********************************************************************************
%**********************************************************************************

\section*{Acknowledgement}

The authors would like to thank Ali Pezeshki for enlightening discussions on this topic. We are also grateful to
the anonymous referees for valuable comments and suggestions.
The first and second author were supported by the grant AFOSR F1ATA00183G003, NSF 1008183, and  DTRA/ NSF 1042701.
Part of this work was completed while the second author visited the Institute of Mathematics at University
of Osnabr\"uck. This author would like to thank this institute for its hospitality and support during his visit.
The second and third author gratefully acknowledge the support of the Institute of Advanced Study through the Park
City Math Institute, where part of this work was completed. The third author also acknowledges the support of the
Hausdorff Center for Mathematics.
The fourth author acknowledges support by DFG Grant SPP-1324, KU 1446/13 and DFG Grant, KU 1446/14. She would like
to thank the Department of Mathematics at University of Missouri for its hospitality and support during her visit,
which enabled completion of this work.

\bibliographystyle{plain}
\bibliography{OSFbibfile}

\end{document}